\documentclass[a4paper,12pt]{article}
\usepackage{amssymb,amsmath,mathrsfs,euscript,amsopn}

\numberwithin{equation}{section}

\newtheorem{lemma}{Lemma}[section]

\newtheorem{theorem}{Theorem}[section]
\newtheorem{proposition}{Proposition}[section]

\newtheorem{definition}{Definition}[section]
\newtheorem{remark}{Remark}[section]

\newenvironment{proof}{\smallskip\noindent{\bf Proof.}\rm}{\hspace*{\fill} $\Box$\medskip}
\newenvironment{proofsketch}{\smallskip\noindent{\bf Sketch of the proof.}\rm}{\hspace*{\fill} $\Box$\medskip}

\newenvironment{prooflemma34}{\smallskip\noindent{\bf Proof of Lemma \ref{Lemma5}.}\rm}{\hspace*{\fill} $\Box$\medskip}
\newenvironment{proofth12}{\smallskip\noindent{\bf Proof of Theorem \ref{Th2}.}\rm}{\hspace*{\fill} $\Box$\medskip}
\newenvironment{proofth14}{\smallskip\noindent{\bf Proof of Theorem \ref{Th4}.}\rm}{\hspace*{\fill} $\Box$\medskip}

\renewcommand{\Re}{\operatorname{Re}}
\renewcommand{\Im}{\operatorname{Im}}
\newcommand{\res}{\operatorname{res}}
\newcommand{\rank}{\operatorname{rank}}
\newcommand{\bq}{\mathbf{q}}
\newcommand{\Ran}{\operatorname{Ran}}
\newcommand{\tH}{\widetilde H}
\newcommand{\tR}{\widetilde R}

\title{Inverse spectral problems for Dirac operators\\ on a finite interval}

\author{
Ya.~V.~Mykytyuk, D.~V.~Puyda\thanks{\emph{Email addresses:} yamykytyuk@yahoo.com (Ya.~V.~Mykytyuk), dpuyda@gmail.com (D.~V.~Puyda)}\\
\emph{Ivan Franko National University of Lviv}\\
\emph{1 Universytetska str., Lviv, 79000, Ukraine}
}

\date{}

\begin{document}

\maketitle

\begin{abstract}
%% Text of abstract
We consider the direct and inverse spectral problems for Dirac
operators generated by the differential expressions
$$
\mathfrak
t_q:=\frac{1}{i}\begin{pmatrix}I&0\\0&-I\end{pmatrix}\frac{d}{dx}+
\begin{pmatrix}0&q\\q^*&0\end{pmatrix}
$$
and some separated boundary conditions. Here $q$ is an $r\times r$
matrix-valued function with entries belonging to
$L_2((0,1),\mathbb C)$ and $I$ is the identity $r\times r$ matrix.
We give a complete description of the spectral data (eigenvalues
and suitably introduced norming matrices) for the operators under
consideration and suggest a method for reconstructing the
potential $q$ from the corresponding spectral data.
\end{abstract}

\section{Introduction} \label{}
Direct and inverse spectral problems for Dirac and
Sturm--Liouville operators are the objects of interest in plenty
of papers. In 1966, M. Gasymov and B. Levitan solved the inverse
spectral problem for Dirac operators on a half-line by using the
spectral function \cite{GasLev1} and the scattering phase
\cite{GasLev2}. Their investigations were continued and further
developed in many directions.

By now, the direct and inverse spectral problems for Dirac
operators with potentials from different classes have been solved.
For instance, the Dirac operators on a finite interval with
continuous potentials were considered in \cite{GasDzab},
\cite{Malamud2} (reconstructing from two spectra), the ones on a
half-line were treated in \cite{Malamud} (complete description of
the spectral measures and the reconstruction procedure). The case
of potentials belonging to $L_p(0,1)$, $p\ge1$, was considered in
\cite{MykDirac} (reconstructing from two spectra and from one
spectrum and the norming constants based on the Krein equation).

The Weyl--Titchmarsh $m$-functions were used in \cite{ClarkGesz},
\cite{KisMakGesz} to recover the Dirac operators acting in
$L_2(\mathbb R_+,\mathbb C^{2r})$. More general canonical systems
on $\mathbb R$ were considered in \cite{Sakh1}, \cite{Sakh2}. The
matrix-valued Weyl--Titchmarsh functions were recently used in
\cite{Korot} for the characterization of vector-valued
Sturm--Liouville operators on the unit interval.

There are many other interesting papers concerning the direct and
inverse spectral problems for Dirac and Sturm--Liouville operators
besides those mentioned here. We refer the reader to the extensive
bibliography cited in \cite{Malamud2}--\cite{sturm} for further
results on that subject.

The aim of the present paper is to extend the results of the
recent paper \cite{sturm} by Ya. Mykytyuk and N. Trush concerning
the inverse spectral problems for Sturm--Liouville operators with
matrix-valued potentials to the case of Dirac operators on a
finite interval with square-summable potentials.

\subsection{Setting of the problem} \label{}
Let $M_r$ denote the Banach algebra of $r\times r$ matrices with
complex entries, which we identify with the Banach algebra of
linear operators $\mathbb C^r\to\mathbb C^r$ endowed with the
standard norm. We write $I=I_r$ for the unit element of $M_r$ and
$M_r^+$ for the set of all matrices $A\in M_r$ such that
$A=A^*\ge0$. Also we use the notations
$$
\mathbb H:=L_2((0,1),\mathbb C^r)\times L_2((0,1),\mathbb
C^r),\quad \mathfrak Q:=L_2((0,1),M_r).
$$

Let $q\in\mathfrak Q$. Denote
\begin{equation}\label{tqdef}
\vartheta:=\frac{1}{i}\begin{pmatrix}I&0\\0&-I\end{pmatrix},\quad
\bq:=\begin{pmatrix}0&q\\q^*&0\end{pmatrix}
\end{equation}
and consider the differential expression
\begin{equation}\label{DiffExpr}
\mathfrak t_q:=\vartheta\frac{d}{dx}+\bq
\end{equation}
on the domain $ D(\mathfrak t_q)=\{y=(y_1, \ y_2)^\top \ | \
y_1,y_2\in W_2^1((0,1),\mathbb C^r) \}$, where $W_2^1$ is the
Sobolev space. The object of our investigation is a self-adjoint
Dirac operator $T_q$ that is generated by the differential
expression (\ref{DiffExpr}) and the separated boundary conditions
$y_1(0)=y_2(0)$, $y_1(1)=y_2(1)$:
$$
T_qy=\mathfrak t_q(y), \quad D(T_q)=\{y\in D(\mathfrak t_q) \ | \
y_1(0)=y_2(0), \ y_1(1)=y_2(1)\}.
$$
The function $q\in\mathfrak Q$ will be conventionally called the
\emph{potential} of $T_q$.

The spectrum $\sigma(T_q)$ of the operator $T_q$ consists of
countably many isolated real eigenvalues of finite multiplicity,
accumulating only at $+\infty$ and $-\infty$. We denote by
$\lambda_j(q)$, $j\in\mathbb Z$, the pairwise distinct eigenvalues
of the operator $T_q$ labeled in increasing order so that
$\lambda_0(q)\le0<\lambda_1(q)$:
$$
\sigma(T_q)=\{\lambda_j(q)\}_{j\in\mathbb Z}.
$$

Denote by $m_q$ the Weyl--Titchmarsh function of the operator
$T_q$ that is defined as in \cite{ClarkGesz}. The function $m_q$
is a matrix-valued meromorphic Herglotz function (i.e. $\Im
m_q(\lambda)\ge0$ whenever $\Im\lambda>0$), and
$\{\lambda_j(q)\}_{j\in\mathbb Z}$ is the set of its poles. We put
$$
\alpha_j(q):=-\underset{\lambda=\lambda_j(q)}\res
m_q(\lambda),\quad j\in\mathbb Z,
$$
and call $\alpha_j(q)$ the \emph{norming matrix} of the operator
$T_q$ corresponding to the eigenvalue $\lambda_j(q)$. Note that
the multiplicity of the eigenvalue $\lambda_j(q)$ of $T_q$ equals
$\rank\alpha_j(q)$ and that $\alpha_j(q)\ge0$ for all $j\in\mathbb
Z$.

We call the collection $\mathfrak
a_q:=((\lambda_j(q),\alpha_j(q)))_{j\in\mathbb Z}$ the
\emph{spectral data} of the operator $T_q$, and the matrix-valued
measure
$$
\mu_q:=\sum\limits_{j=-\infty}^{\infty}\alpha_j(q)\delta_{\lambda_j(q)}
$$
is called its \emph{spectral measure}. Here $\delta_\lambda$ is
the Dirac delta-measure centered at the point $\lambda$. In
particular, if $q\equiv0$ then
$$
\mu_0=\sum\limits_{n=-\infty}^{\infty}I\delta_{\pi n}.
$$

The aim is to give a complete description of the class $\mathfrak
A:=\{\mathfrak a_q \ | \ q\in\mathfrak Q\}$ of spectral data for
Dirac operators under consideration, which is equivalent to the
description of the class $\mathfrak M:=\{\mu_q \ | \ q\in\mathfrak
Q\}$ of spectral measures, and to suggest an efficient method of
reconstructing the potential $q$ from the corresponding spectral
data $\mathfrak a_q$.

\subsection{{Main results}\label{subsec.12}}
We start from the description of spectral data for operators under
consideration. In what follows $\mathfrak a$ will stand for an
arbitrary sequence $((\lambda_j,\alpha_j))_{j\in\mathbb Z}$, in
which $(\lambda_j)_{j\in\mathbb Z}$ is a strictly increasing
sequence of real numbers such that $\lambda_0\le0<\lambda_1$ and
$\alpha_j$ are non-zero matrices in $M_r^+$. By $\mu^{\mathfrak
a}$ we denote the matrix-valued measure given by
\begin{equation}\label{MuAdef}
\mu^{\mathfrak a}:=\sum\limits_{j=-\infty}^\infty
\alpha_j\delta_{\lambda_j}.
\end{equation}

We partition the real axis into pairwise disjoint intervals
$\Delta_n$, $n\in\mathbb Z$:
$$
\Delta_n:=\left(\pi n-\tfrac{\pi}{2},\pi n+\tfrac{\pi}{2}\right],\
n \in \mathbb Z.
$$
A complete description of the class $\mathfrak A$ is given by the
following theorem.

\begin{theorem}\label{Th1}
In order that a sequence $\mathfrak
a=((\lambda_j,\alpha_j))_{j\in\mathbb Z}$ should belong to
$\mathfrak A$ it is necessary and sufficient that the following
conditions are satisfied:
\begin{itemize}
\item[$(A_1)$]$\sup\limits_{n\in \mathbb
Z}\sum\limits_{\lambda_j\in\Delta_n}1<\infty$, \
$\sum\limits_{n\in \mathbb
Z}\sum\limits_{\lambda_j\in\Delta_n}|\lambda_j-\pi n|^2<\infty$, \
$\sum\limits_{n\in\mathbb
Z}\|I-\sum\limits_{\lambda_k\in\Delta_n}\alpha_k\|^2<\infty$;

\item[$(A_2)$]$\exists N_0\in\mathbb N \ \ \forall
N\in \mathbb N:\ (N\ge N_0)\Rightarrow
\sum\limits_{n=-N}^{N}\sum\limits_{\lambda_j\in\Delta_n}\rank
\alpha_j=(2N+1)r$;

\item[$(A_3)$]the system of functions $\{ de^{i\lambda_jt}\ | \ j\in\mathbb{Z},\
d\in\mathrm{Ran}\ \alpha_j \}$ is complete in $L_2((-1,1),\mathbb
C^r)$.
\end{itemize}
\end{theorem}

By definition, every $\mathfrak a\in\mathfrak A$ forms the
spectral data for Dirac operator $T_q$ with some $q\in\mathfrak
Q$. It turns out that this spectral data determine the potential
$q$ uniquely:

\begin{theorem}\label{Th2}
The mapping $\mathfrak Q\owns q\mapsto\mathfrak a=\mathfrak
a_q\in\mathfrak A$ is bijective.
\end{theorem}

We base our algorithm of reconstructing the potential $q$ from the
corresponding spectral data $\mathfrak a_q$ on Krein's
accelerant method.
\begin{definition}
We say that a function $H\in L_2((-1,1),M_r)$ is an accelerant if
for every $a\in[0,1]$ the integral equation
$$
f(x)+\int\limits_0^aH(x-t)f(t)dt=0
$$
has only trivial solution in $L_2((0,1),\mathbb C^r)$. We denote
the set of accelerants by $\mathfrak H_2$ and endow it with the
metric of the space $L_2((-1,1),M_r)$.
\end{definition}

We set $\mathfrak H_2^s:=\{H\in\mathfrak H_2 \ | \ H(x)^*=H(-x)\
\textrm{a.e. for}\ x\in(-1,1)\}$.

The spectral data of the operator $T_q$ generate Krein's
accelerant as explained in the following theorem.

\begin{theorem}\label{Th3}
Take a sequence $\mathfrak a=((\lambda_j,\alpha_j))_{j\in\mathbb
Z}$ satisfying the asymptotics $(A_1)$, and set
$\mu:=\mu^{\mathfrak a}$. Then the limit
\begin{equation}\label{HMudef}
H_\mu(x)=\lim\limits_{n\to\infty}\int\limits_{-\pi\left(n-\frac{1}{2}\right)}^
{\pi\left(n+\frac{1}{2}\right)}e^{2i\lambda
x}d(\mu-\mu_0)(\lambda)
\end{equation}
exists in the topology of the space $L_2((-1,1),M_r)$. If, in
addition, $(A_3)$ holds, then the function $H_\mu$ is an
accelerant and belongs to $\mathfrak H_2^s$.
\end{theorem}

By virtue of Theorem \ref{Th1}, any $\mathfrak a\in\mathfrak A$
satisfies the conditions $(A_1)-(A_3)$. In addition, if
$q\in\mathfrak Q$ and $\mathfrak a=\mathfrak a_q$, then
$\mu^{\mathfrak a}=\mu_q$. Therefore according to Theorem
\ref{Th3} we can define the mapping
$q\mapsto\Upsilon(q):=H_{\mu_q}$ acting from $\mathfrak Q$ to
$\mathfrak H_2^s$, and in order to solve the inverse spectral
problem for the operator $T_q$ we have to find the inverse mapping
$\Upsilon^{-1}$. As in \cite{sturm}, it can be done using the
Krein equation.

It is known that for all $H\in\mathfrak H_2$ the Krein equation
\begin{equation}\label{KreinEq}
R(x,t)+H(x-t)+\int\limits_0^xR(x,s)H(s-t)ds=0,\quad
(x,t)\in\Omega^+,
\end{equation}
where $\Omega^+:=\{(x,t) \ | \ 0\le t\le x\le 1\}$, has a unique
solution $R_H$ in the class $L_2(\Omega^+,M_r)$. Moreover, if we
extend $R_H$ by zero to the triangle $\Omega^-:=\{(x,t) \ | \ 0\le
x<t\le 1\}$, we obtain that $R_H\in G_2(M_r)$ (see \ref{add.1}).

Thus we can define the mapping $\Theta:\mathfrak H_2^s\to
\mathfrak Q$ given by
\begin{equation}\label{ThetaDef}
\Theta(H):=iR_H(\cdot,0).
\end{equation}
The following theorem explains how to solve the inverse spectral
problem for the operator $T_q$.

\begin{theorem}\label{Th4}
$\Upsilon^{-1}=\Theta$. In particular, if $q\in\mathfrak Q$,
$\mathfrak a=\mathfrak a_q$, $\mu=\mu^{\mathfrak a}$, then
\begin{equation}\label{QafterTheta}
q=\Theta(H_\mu).
\end{equation}
\end{theorem}

According to this theorem the reconstruction algorithm can proceed
as follows. Given $\mathfrak a\in\mathfrak A$ we construct the
matrix-valued measure $\mu:=\mu^{\mathfrak a}$ via (\ref{MuAdef}),
which generates the accelerant $H:=H_\mu$ via (\ref{HMudef}).
Solving the Krein equation (\ref{KreinEq}) we find the function
$R_H$, which gives us $q$ via the formulas (\ref{QafterTheta}) and
(\ref{ThetaDef}). That $q$ is the function looked for follows from
the fact that the Dirac operator $T_q$ has the spectral data
$\mathfrak a$ we have started with.

We visualize the reconstruction algorithm by means of the
following diagram:
$$
\mathfrak a \xrightarrow[s_1]{(\ref{MuAdef})} \mu^{\mathfrak
a}=:\mu \xrightarrow[s_2]{(\ref{HMudef})}H_\mu=:H
\xrightarrow[s_3]{(\ref{KreinEq})}R_H
\xrightarrow[s_4]{(\ref{ThetaDef})}\Theta(H)=q.
$$
Here $s_j$ denotes the step number $j$. Steps $s_1$, $s_2$, $s_4$
are trivial. The basic and non-trivial step is $s_3$.

\begin{remark}
One can also consider the case of more general separated
self-adjoint boundary conditions. Denote by $T_{q,a,b}$ the
operator generated by the differential expression
(\ref{DiffExpr}) and the boundary conditions
$$
ay(0)=0,\quad by(1)=0,
$$
where $a$ and $b$ are $r\times2r$ matrices with complex entries such
that (see \cite{ClarkGesz})
$$
aa^*=bb^*=I,\quad a\vartheta a^*=b\vartheta b^*=0.
$$
For the operator $T_{q,a,b}$, the analogues
of Theorems \ref{Th1}--\ref{Th4} can be proved, but their
formulations are more complicated since the spectrum of the
non-perturbed operator $T_{0,a,b}$ has a more involved structure.
Namely, it consists of $2r$ eigenvalue sequences of the form $(\lambda^0_j+2\pi k)_{k\in\mathbb{Z}}$, $j=1,\dots,2r$, counting multiplicities. The authors plan to
consider the case of general (not necessarily separated) boundary
conditions in a forthcoming paper.
\end{remark}

\section{Direct spectral analysis}
In this section we study the properties of the spectral data for
operators under consideration.

\subsection{Basic properties of the operator $T_q$}
Here we prove self-adjointness of $T_q$, construct its resolvent
and the resolution of identity.

Let $\lambda\in\mathbb C$. For an arbitrary $q\in\mathfrak Q$
denote by $u_q=u_q(\cdot,\lambda)\in W_2^1((0,1),M_{2r})$ a
solution of the Cauchy problem
\begin{equation}\label{probl1}
\vartheta \tfrac{d}{dx}u+\bq u=\lambda u,\quad
u(0,\lambda)=I_{2r},
\end{equation}
where $\vartheta$ and $\bq$ are defined via (\ref{tqdef}). Note
that if $q\equiv0$ then
\begin{equation}\label{Phi0}
u_0(x,\lambda)=\begin{pmatrix}e^{i\lambda x}I&0\\0&e^{-i\lambda
x}I\end{pmatrix}.
\end{equation}
Denote
\begin{equation}\label{PhiPsiDef}
\varphi_q(\cdot,\lambda):=u_q(\cdot,\lambda)\vartheta a^*,\quad
\psi_q(\cdot,\lambda):=u_q(\cdot,\lambda)a^*,
\end{equation}
where
$$
a:=\tfrac{1}{\sqrt{2}}\begin{pmatrix}I,&-I\end{pmatrix},
$$
and set $s(\lambda,q):=a\varphi_q(1,\lambda)$,
$c(\lambda,q):=a\psi_q(1,\lambda)$,
$m_q(\lambda):=-s(\lambda,q)^{-1}c(\lambda,q)$. We call $m_q$
\emph{the Weyl--Titchmarsh} function of the operator $T_q$.

Some basic properties of the objects just introduced are described
in the following lemma.
\begin{lemma}\label{Th5}
\begin{itemize}
\item[$(i)$]For every $q\in\mathfrak Q$ there exists a unique
matrix-valued function $K_q\in G_2^+(M_{2r})$ such that for any
$\lambda\in\mathbb C$ and $x\in[0,1]$,
\begin{equation}\label{varphiRepr}
\varphi_q(x,\lambda)=\varphi_0(x,\lambda)+\int\limits_0^x
K_q(x,s)\varphi_0(s,\lambda)ds;
\end{equation}
\item[$(ii)$]the mapping $\mathfrak Q\owns q\mapsto K_q\in
G_2^+(M_{2r})$ is continuous;
\item[$(iii)$]the matrix-valued functions $\lambda\mapsto
s(\lambda,q)$ and $\lambda\mapsto c(\lambda,q)$ are entire and
allow the representations
\begin{equation}\label{slambdarepr}
s(\lambda,q)=(\sin\lambda) I+\int\limits_{-1}^{1}e^{i\lambda
t}g_1(t)dt,\quad c(\lambda,q)=(\cos\lambda)
I+\int\limits_{-1}^{1}e^{i\lambda t}g_2(t)dt,
\end{equation}
where $g_1$ and $g_2$ are some (depending on $q$) functions from
the space $L_2((-1,1),M_r)$;
\item[$(iv)$]for every $q\in\mathfrak Q$ the following relation holds:
\begin{equation}\label{PhiPsiRel}
-\psi_q(x,\lambda)\varphi_q(x,\overline\lambda)^*+\varphi_q(x,\lambda)\psi_q(x,\overline\lambda)^*\equiv\vartheta.
\end{equation}
\end{itemize}
\end{lemma}

\begin{proof}
Let us fix $q\in\mathfrak Q$ and set $q_1:=-\Im
q=-\frac{1}{2}(q-q^*)$, $q_2:=\Re q=\frac{1}{2}(q+q^*)$. Consider
the Cauchy problem
$$
B\tfrac{d}{dx}v+Qv=\lambda v, \quad v(0,\lambda)=I_{2r},
$$
where
$$
B:=\begin{pmatrix}0&I\\-I&0\end{pmatrix},\quad
Q:=\begin{pmatrix}q_1&q_2\\q_2&-q_1\end{pmatrix}.
$$
It follows from \cite{cauchy} that this problem has a unique
solution $v_q=v_q(\cdot,\lambda)$ in $W_2^1((0,1),M_{2r})$ and
that $v_q(\cdot,\lambda)$ can be represented in the form
\begin{equation}\label{Vform}
 v_q(x,\lambda)=e^{-\lambda xB}
        + \int\limits_0^x P^+(x,s)e^{-\lambda (x-2s) B}ds
        + \int\limits_0^x P^-(x,s)e^{\lambda (x-2s) B}ds,
\end{equation}
where $ e^{xB}=(\cos x)I_{2r}+(\sin x)B$.

Note that $\vartheta=W^{-1}BW$ and $\bq=W^{-1}Q W$, where $W$ is
the unitary matrix
$$
W=\frac{1}{\sqrt{2}}\begin{pmatrix}I&-iI\\-iI&I\end{pmatrix}.
$$
Therefore the function $u_q(\cdot,\lambda)=W^{-1}v_q(\cdot,\lambda)W$
solves the Cauchy problem (\ref{probl1}).

Note that
$$
e^{xB}J=Je^{-xB},\quad J=\begin{pmatrix}0&I\\I&0\end{pmatrix}.
$$
Using now (\ref{Vform}) and performing some calculations we easily
obtain that
$$
\varphi_q(x,\lambda)=\varphi_0(x,\lambda)+\int\limits_0^x
K_q(x,s)\varphi_0(s,\lambda)ds,
$$
where $K_q(x,t)=W^{-1}P_Q(x,t)W$ and
$$
P_Q(x,t)=\tfrac{1}{2}\left\{P^+\left(x,\tfrac{x-t}{2}\right)+P^+\left(x,\tfrac{x+t}{2}\right)J
+P^-\left(x,\tfrac{x-t}{2}\right)J+P^-\left(x,\tfrac{x+t}{2}\right)\right\}.
$$
It follows from \cite{cauchy} that the function $P_Q$ belongs to $G_2^+(M_{2r})$ and that the mapping $ L_2((0,1),M_{2r})\owns
Q\mapsto P_Q\in G_2^+(M_{2r}) $ is continuous. Therefore the first
two statements of the present lemma will be proved if we prove the
uniqueness of the representation (\ref{varphiRepr}), but this can
be easily done repeating the proof given in \cite{cauchy}.

Now let us prove $(iii)$. By virtue of the definition of
$s(\lambda,q)$ and the representation (\ref{varphiRepr}) we obtain
that
$$
s(\lambda,q)=(\sin\lambda)I+\int\limits_0^1aK_q(1,s)\varphi_0(s,\lambda)ds
$$
and simple calculations yield the formula for $s(\lambda)$ in
(\ref{slambdarepr}) with some $g_1\in L_2((-1,1),M_r)$. Having
noted that
$\psi_q(x,\lambda)=u_q(x,\lambda)a^*=W^{-1}v_q(x,\lambda)Wa^*$ and
taking into consideration (\ref{Vform}) we can analogously obtain
the formula for $c(\lambda)$.

It remains to prove $(iv)$. A direct verification shows that
$$
\tfrac{d}{dx}\{u_q(x,\overline\lambda)^*\vartheta
u_q(x,\lambda)\}\equiv0,
$$
and therefore we obtain the relation
$u_q(x,\overline\lambda)^*\vartheta
u_q(x,\lambda)\equiv\vartheta$. From this equality we obtain that
$\vartheta u_q(x,\lambda)\vartheta u_q(x,\overline\lambda)^*\equiv -I_{2r}$, and thus
$u_q(x,\lambda)\vartheta
u_q(x,\overline\lambda)^*\equiv\vartheta$. Having noted that
$\vartheta=a^*a\vartheta+\vartheta a^*a$ we conclude that
$$
u_q(x,\lambda)a^*a\vartheta
u_q(x,\overline\lambda)^*+u_q(x,\lambda)\vartheta a^*a
u_q(x,\overline\lambda)^*\equiv\vartheta,
$$
which proves the relation (\ref{PhiPsiRel}).
\end{proof}

For $\lambda\in\mathbb C$ denote by $\Phi_q(\lambda)$ the operator
acting from $\mathbb C^r$ to $\mathbb H$ by the formula
\begin{equation}\label{PhiOperDef}
[\Phi_q(\lambda)c](x):=\varphi_q(x,\lambda)c.
\end{equation}
Taking into consideration (\ref{varphiRepr}) we obtain that
\begin{equation}\label{PhiKrel}
\Phi_q(\lambda)=(\mathscr I+\mathscr K_q)\Phi_0(\lambda),\quad
\lambda\in\mathbb C,
\end{equation}
where $\mathscr K_q$ is an integral operator with kernel $K_q$ and
$\mathscr I$ is the identity operator in $\mathscr B(\mathbb H)$,
which is the algebra of bounded linear operators acting in
$\mathbb H$. Note that since $K_q$ belongs to $G_2^+(M_{2r})$, the
operator $\mathscr K_q$ belongs to $\mathscr G_2^+(M_{2r})$ (see
\ref{add.1}), and hence it is a Volterra operator (see
\cite{kreinvolterra}).

Some properties of the operators $\Phi_q(\lambda)$ and the
Weyl--Titchmarsh function $m_q(\lambda)$ are formulated in the
following lemma.

\begin{lemma}\label{Lemma1}
Let $q\in\mathfrak Q$. Then the following statements hold:
\begin{itemize}
\item[$(i)$]the operator function $\lambda\mapsto\Phi_q(\lambda)$
is analytic in $\mathbb C$; moreover, for $\lambda\in\mathbb C$
\begin{equation}\label{KerRanPhi}
\ker\Phi_q(\lambda)=\{0\},\quad \Ran\Phi_q(\lambda)^*=\mathbb C^r,
\end{equation}
\begin{equation}\label{KerPhiRel}
\ker(T_q-\lambda\mathscr I)=\Phi_q(\lambda)\ker s(\lambda,q).
\end{equation}
\item[$(ii)$]the operator functions $\lambda\mapsto
s(\lambda,q)^{-1}$ and
$$
\lambda\mapsto
m_q(\lambda)=-s(\lambda,q)^{-1}c(\lambda,q)
$$
are meromorphic in
$\mathbb C$; moreover, $m_0(\lambda)=-\cot\lambda I$ and
\begin{equation}\label{mAsymp}
\|m_q(\lambda)+\cot\lambda I\|=o(1)
\end{equation}
as $\lambda\to\infty$ within the domain $\mathcal O=\{z\in\mathbb
C \ | \ \forall n\in\mathbb Z \ |z-\pi n|>1\}$.
\end{itemize}
\end{lemma}

\begin{proof}
The proof of this lemma is analogous to the proof of Lemma 2.3 in
\cite{sturm}.
\end{proof}

Finally, basic properties of the operator $T_q$ are described in
the following theorem.

\begin{theorem}\label{Th6}
Let $q\in\mathfrak Q$. Then the following statements hold:
\begin{itemize}
\item[$(i)$]the operator $T_q$ is self-adjoint;
\item[$(ii)$]the spectrum $\sigma(T_q)$ of $T_q$ consists of
isolated real eigenvalues and
$$
\sigma(T_q)=\{\lambda \ | \ \ker s(\lambda,q)\neq\{0\}\};
$$
\item[$(iii)$]let $\lambda_j=\lambda_j(q)$ and let $P_{j,q}$ be
the orthogonal projector on $\ker(T_q-\lambda\mathscr I)$, then
$$
\sum\limits_{j=-\infty}^{\infty}P_{j,q}=\mathscr I;
$$
\item[$(iv)$]the norming matrices $\alpha_j=\alpha_j(q)$ satisfy
the relations $\alpha_j\ge0$, $j\in\mathbb Z$; moreover, for all
$j\in\mathbb Z$ we have
$$
P_{j,q}=\Phi_q(\lambda_j)\alpha_j\Phi_q^*(\lambda_j),
$$
where $\Phi_q^*(\lambda):=[\Phi_q(\lambda)]^*$.
\end{itemize}
\end{theorem}

\begin{proof}
A direct verification shows that the operator $T_q$ is symmetric.
Take an arbitrary $\lambda$ such that the matrix $s(\lambda,q)$ is
non-singular, and let $f\in\mathbb H$. Then the function
$$
g(x)=[\mathscr T(\lambda)f](x):=\psi_q(x,\lambda)\int\limits_0^x
\varphi_q(t,\overline\lambda)^*f(t)dt+
\varphi_q(x,\lambda)\int\limits_x^1
\psi_q(t,\overline\lambda)^*f(t)dt
$$
belongs to the domain of differential expression $\mathfrak t_q$
and solves the Cauchy problem
$$
\mathfrak t_q(g)=\lambda g+f,\quad ag(0)=0,
$$
as can be directly verified using (\ref{PhiPsiRel}). A generic
solution of this problem takes the form
$h=\varphi_q(\cdot,\lambda)c+\mathscr T(\lambda) f$, $c\in\mathbb
C^r$. The choice
$$
c=m_q(\lambda)\int\limits_0^1
\varphi_q(t,\overline\lambda)^*f(t)dt
$$
gives that $ah(1)=0$, i.e. the boundary conditions
$h_1(0)=h_2(0)$, $h_1(1)=h_2(1)$ are satisfied. This implies that
$\lambda$ is a resolvent point of the operator $T_q$, and the
resolvent of $T_q$ is given by
$$
(T_q-\lambda\mathscr
I)^{-1}=\Phi_q(\lambda)m_q(\lambda)\Phi_q^*(\overline\lambda)+\mathscr
T(\lambda).
$$
Since $\mathscr T(\lambda)$ is a Hilbert--Schmidt operator, the
operator $T_q$ has a compact resolvent, and therefore the
statements $(i)-(iii)$ are proved.

Recall that $-\alpha_j(q)$ is a residue of the Weyl--Titchmarsh
function at the point $\lambda_j=\lambda_j(q)$, $j\in\mathbb Z$.
Taking $\varepsilon>0$ small enough we obtain that
$$
P_{j,q}=-\frac{1}{2\pi i}\oint\limits_{|\lambda-\lambda_j|
=\varepsilon} (T_q-\lambda\mathscr I)^{-1}d\zeta
=\Phi_q(\lambda_j) \alpha_j(q)\Phi_q^*(\lambda_j)
$$
for every $j\in\mathbb Z$.

By virtue of (\ref{KerRanPhi}) we obtain that $\alpha_j(q)\ge0$
for all $j\in\mathbb Z$, and the statement $(iv)$ is also proved.
\end{proof}

\subsection{Description of the spectral data: the necessity part}
Here we show that if $q\in\mathfrak Q$, then the spectral data
$\mathfrak a_q$ satisfy the conditions $(A_1)-(A_3)$, which is the
necessity part of Theorem \ref{Th1}.

\subsubsection{The condition $(A_1)$}
In the sequel we shall use the following notations. If
$(\lambda_{j})_{j\in\mathbb Z}$ is a strictly increasing sequence
of non-negative real numbers and $(\alpha_j)_{j\in\mathbb Z}$ is a
sequence in $M^+_r$, then
\begin{equation}\label{eq.215}
          \beta_n:= I-\sum\limits_{\lambda_k\in\Delta_n}\alpha_k, \qquad
          \widetilde\lambda_j:=\lambda_j-\pi n, \quad \lambda_j\in\Delta_n,\quad
          n\in\mathbb Z,
\end{equation}
with $\Delta_n$ being defined in Subsection \ref{subsec.12}.

We start from the condition $(A_1)$, which describes the
asymptotics of spectral data.

\begin{theorem}\label{a1th1}
Let $q\in\mathfrak Q.$  Then for the sequence $\mathfrak
a=\mathfrak a_q$ the condition $(A_1)$ holds.
\end{theorem}

\begin{proofsketch}
The proof of this theorem is analogous to the proof in
\cite{sturm}, and therefore we give here only its sketch. Let
$q\in\mathfrak Q$ and $\lambda_j=\lambda_j(q)$,
$\alpha_j=\alpha_j(q)$ for $j\in\mathbb Z$. The eigenvalues
$\lambda_j$ are zeros of the sine-type function $\lambda\mapsto
s(\lambda)$
 (see\eqref{slambdarepr}) that belongs
 to the following class of functions $ \mathbb C\to
 M_r$:
$$
        \mathcal{F}_f(\lambda):=\sin{\lambda}I +
        \int_{-1}^1 f(t)e^{i\lambda t}\,d t, \qquad
        \lambda\in\mathbb C,
$$
where $f\in L_2((-1,1),M_r)$. It is shown in \cite{zeros} that the
set of zeros of a function $\det \mathcal{F}_f$, with ${\mathcal
F}_f$ as above, can be indexed (counting multiplicities) by the
set $\mathbb Z$ so that the corresponding sequence
$(\omega_n)_{n\in\mathbb Z}$ of its zeros has the asymptotics
$$
         \omega_{kr+j}=\pi k +\widehat\omega_{j,k},\qquad
         k\in\mathbb{Z},\quad
         j=0,\dots,r-1,
$$
where the sequences $(\widehat\omega_{j,k})_{k\in\mathbb{Z}}$
belong to $\ell_2(\mathbb Z)$. Therefore,
\begin{equation}\label{eq.216}
  \sup\limits_{n\in\mathbb Z}\sum\limits_{\lambda_j\in\Delta_n}
           1<\infty, \qquad        \sum_{n\in\mathbb Z}\sum\limits_{\lambda_j\in\Delta_n}|
          \widetilde\lambda_j|^2<\infty,
\end{equation}
and thus it is left to prove only that (see (\ref{eq.215}))
$$
\sum\limits_{n=-\infty}^\infty\|\beta_n\|^2<\infty.
$$
It can be done in exactly the same way as in \cite{sturm}.
\end{proofsketch}

\subsubsection{The condition $(A_2)$}
We start from proving the following lemma, which is an analogue of
Lemma 2.12 in \cite{sturm}.

\begin{lemma}\label{Lemma3}
Assume that $q\in\mathfrak Q$, and let $\mathfrak a$ be a
collection satisfying the asymptotics $(A_1)$. For $j\in\mathbb Z$
set $\hat P_j:=\Phi_q(\lambda_j)\alpha_j\Phi_q^*(\lambda_j)$. Then
the series $\sum_{j\in\mathbb Z}\hat P_j$ converges in the strong
operator topology and
\begin{equation}\label{PSeries}
\sum\limits_{n=-\infty}^\infty\|P_{n,0}-\sum\limits_{\lambda_j\in\Delta_n}\hat
P_j\|^2<\infty.
\end{equation}
\end{lemma}

\begin{proofsketch}
Let the assumptions of the present lemma hold, and let $\mathfrak
a=((\lambda_j,\alpha_j))_{j\in\mathbb Z}$. Using (\ref{PhiKrel})
and the fact that $\mathscr K_q$ is a Hilbert--Schmidt operator,
it can be observed that
$$
\sum\limits_{n\in\mathbb Z}\sum\limits_{\lambda_j\in
\Delta_n}\|\Phi_q(\pi n)-\Phi_0(\pi n)\|^2<\infty.
$$
Since $\|\Phi_q(\lambda_j)-\Phi_q(\pi n)\|\le
C|\widetilde\lambda_j|$ ($\lambda_j\in \Delta_n$, $n\in\mathbb Z$) for
some $C>0$,
\begin{equation}\label{PhiSeries}
\sum\limits_{n\in\mathbb Z}\sum\limits_{\lambda_j\in
\Delta_n}\|\Phi_q(\lambda_j)-\Phi_0(\pi n)\|^2<\infty.
\end{equation}
From ~\eqref{PhiSeries} we easily obtain that
\begin{equation}\label{PhiFSeries}
\sum\limits_{j\in\mathbb Z}\|\Phi_q^*(\lambda_j)f\|^2<\infty
\end{equation}
for all $f\in\mathbb H$. Indeed, it is enough to note that $
\sum_{n\in\mathbb Z}\|\Phi_0(\pi n)^*f\|^2= \|f\|^2$, $f\in\mathbb
H$, and that $\sup\limits_{n\in \mathbb
Z}\sum_{\lambda_j\in\Delta_n}1<\infty$.

Taking into account that the sequence $(\alpha_j)_{j\in\mathbb Z}$
is bounded we conclude that $ \sum_{j\in\mathbb
Z}\|\alpha_j\Phi^*(\lambda_j)f\|^2<\infty$, $f\in\mathbb H$.
Moreover, it can also be shown that for every sequence $c\in
l_2(\mathbb Z,\mathbb C^r)$ the series $\sum_{j\in\mathbb
Z}\Phi_q(\lambda_j)c_j$ is convergent, which justifies the
convergence of $\sum_{j\in\mathbb Z}\hat P_j$.

Now let us prove (\ref{PSeries}). Recall that $P_{n,0}=\Phi_0(\pi
n)\Phi_0^*(\pi n)$. By virtue of the definition of $\beta_n$ we
 obtain that
$$
P_{n,0}=\Phi_0(\pi n)\beta_n\Phi_0^*(\pi
n)+\sum\limits_{\lambda_j\in\Delta_n}\Phi_0(\pi
n)\alpha_j\Phi_0^*(\pi n),
$$
and thus we can write
$$
P_{n,0}-\sum\limits_{\lambda_j\in\Delta_n} \hat P_j=\Phi_0(\pi
n)\beta_n\Phi_0^*(\pi
n)+\sum\limits_{\lambda_j\in\Delta_n}[\Phi_0(\pi
n)\alpha_j\Phi_0^*(\pi
n)-\Phi_q(\lambda_j)\alpha_j\Phi_q^*(\lambda_j)].
$$
Thus, since the sequences $(\Phi_q(\lambda_j))$ and $(\alpha_j)$ are
bounded, we obtain that
$$
\|P_{n,0}-\sum\limits_{\lambda_j\in\Delta_n}\hat P_j\|^2\le
C_1\|\beta_n\|^2+C_2\sum\limits_{\lambda_j\in\Delta_n}\|\Phi_q(\lambda_j)-\Phi_0(\pi
n)\|^2,
$$
where $C_1$ and $C_2$ are non-negative constants independent of
$n$. Taking now into consideration (\ref{PhiSeries}) and $(A_1)$,
we obtain (\ref{PSeries}).
\end{proofsketch}

\noindent The following lemma is proved in \cite{sturm} (Lemma
B.1).

\begin{lemma}\label{Lemma4}
Suppose that H is a Hilbert space. Let $(P_n)_{n=1}^\infty$ and
$(G_n)_{n=1}^\infty$ be sequences of pairwise orthogonal
projectors of finite rank in $H$ such that $\sum_{n=1}^\infty
P_n=\sum_{n=1}^\infty G_n=\mathscr I_H$, where $\mathscr I_H$ is
the identity operator in $H$, and let
$\sum_{n=1}^\infty\|P_n-G_n\|^2<\infty$. Then there exists
$N_0\in\mathbb N$ such that for all $N\ge N_0$,
$$
\sum\limits_{n=1}^N \rank P_n=\sum\limits_{n=1}^N \rank G_n.
$$
\end{lemma}

We use Lemmas \ref{Lemma3} and \ref{Lemma4} to prove $(A_2)$. If
$\mathfrak a=\mathfrak a_q$, then the operators $\hat P_j$,
$j\in\mathbb Z$, from Lemma \ref{Lemma3} coincide with the
orthogonal projectors $P_{j,q}$ corresponding to the eigenvalues
$\lambda_j$ (see Theorem \ref{Th6}). Since $\{P_{j,q}\}$,
$j\in\mathbb Z$, forms a complete system of orthogonal projectors,
by virtue of Lemmas \ref{Lemma3} and \ref{Lemma4} we justify that
$$
\sum\limits_{n=-N}^N \rank P_{n,0}=\sum\limits_{n=-N}^N
\sum\limits_{\lambda_j\in\Delta_n} \rank P_{j,q}
$$
for $N\ge N_0$. Taking into consideration (\ref{KerRanPhi}), we
obtain that $\rank P_{j,q}=\rank \alpha_j$ and $\rank P_{n,0}=r$
for all $j,n\in\mathbb Z$, and thus we justify that the condition
$\mathrm{(A_2)}$ is satisfied.

\subsubsection{The operators $\mathscr U_{\mathfrak a,q}$}
Before proving the condition $(A_3)$ we have to introduce some
operators that play an important role below.

Let $q\in\mathfrak Q$, and let $\mathfrak
a=((\lambda_j,\alpha_j))_{j\in\mathbb Z}$ be any collection
satisfying the asymptotics $(A_1)$. Construct the operator
$\mathscr U_{\mathfrak a,q}:\mathbb H\to\mathbb H$ by the formula
\begin{equation}\label{UaqDef}
\mathscr U_{\mathfrak a,q}:=\sum\limits_{j\in\mathbb
Z}\Phi_q(\lambda_j)\alpha_j\Phi_q^*(\lambda_j).
\end{equation}
By virtue of Lemma \ref{Lemma3} the operator $\mathscr
U_{\mathfrak a,q}$ is continuous, and, since $\alpha_j\ge0$ for
all $j\in\mathbb Z$ (see Theorem \ref{Th6}), it is also
non-negative.

In particular,
\begin{equation}\label{UaqI}
\mathscr U_{\mathfrak a_q,q}=\mathscr I,
\end{equation}
as follows from Theorem \ref{Th6}.

Now we are going to show that the operator $\mathscr U_{\mathfrak
a,q}$ is the sum of the identity one and a compact one. We start
from proving the following lemma.

\begin{lemma}\label{LemmaHconv}
Let $\mathfrak a$ be any collection satisfying the condition
$(A_1)$. Then the limit (\ref{HMudef}) exists in the topology of
the space $L_2((-1,1),M_r)$, and the following relation holds:
\begin{equation}\label{HadjReal}
H_\mu(x)^*=H_\mu(-x).
\end{equation}
\end{lemma}

\begin{proof}
Taking into consideration the definitions of measures $\mu$ and
$\mu_0$ it is easy to observe that the function $H:=H_\mu$ can be
rewritten as
\begin{equation}\label{HSeries}
H(x)=\sum\limits_{n\in\mathbb
Z}\left\{\left(\sum\limits_{\lambda_j\in\Delta_n}e^{2i\lambda_jx}\alpha_j\right)-e^{2i\pi
nx}I\right\},
\end{equation}
and thus we have to show that the series (\ref{HSeries}) is
convergent in $L_2((-1,1),M_r)$.

Note that
\begin{equation}\label{eq.226}
\left(\sum\limits_{\lambda_j\in\Delta_n}e^{2i\lambda_jx}\alpha_j\right)-e^{2i\pi
nx}I= e^{2i\pi nx}\gamma_n(x) +x e^{2i\pi nx}\eta_n-e^{2i\pi
nx}\beta_n,
\end{equation}
where
$$   \gamma_n(x):=\sum\limits_{\lambda_j\in\Delta_n}
(e^{2i\widetilde\lambda_jx}-1-2i\widetilde\lambda_jx)\alpha_j,
\qquad \eta_n:=\sum\limits_{\lambda_j\in\Delta_n}
2i\widetilde\lambda_j\alpha_j,
$$
$\beta_n$ and $\widetilde\lambda_j$ are given by (\ref{eq.215}).
Since the sequence $(\alpha_j)_{j\in\mathbb Z}$ is bounded and
$|e^z-1-z|\le|z|^2e^{|z|}$, $z\in\mathbb C$, in view of the
condition $\mathrm{(A_1)}$ we obtain that
$$  \sum\limits_{n\in\mathbb Z} \sup_{x\in[0,1]}
\|\gamma_n(x)\|< \infty, \qquad \sum\limits_{n\in\mathbb Z}
\|\eta_n\|^2 <\infty, \qquad\sum\limits_{n\in\mathbb Z}
\|\beta_n\|^2<\infty.
$$
Therefore, taking into consideration (\ref{eq.226}) it is easy to
observe that the series (\ref{HSeries}) is convergent in the
topology of the space $L_2((-1,1),M_r).$

The relation (\ref{HadjReal}) follows directly from the formula
(\ref{HSeries}).
\end{proof}

For $H\in L_2((-1,1),M_r)$ denote
\begin{equation}\label{FHdef}
F_H(x,t):=\frac{1}{2}\begin{pmatrix}H\left(\frac{x-t}{2}\right)&H\left(\frac{x+t}{2}\right)\\
H^\sharp\left(\frac{x+t}{2}\right)&H^\sharp\left(\frac{x-t}{2}\right)\end{pmatrix},
\end{equation}
where $H^\sharp(x):=H(-x)$. Note that $F_H\in G_2(M_{2r})$.

\begin{proposition}
Let $\mathfrak a$ be any collection satisfying the asymptotics
$(A_1)$, and set $\mu:=\mu^{\mathfrak a}$, $H:=H_\mu$. Then
\begin{equation}\label{UF}
\mathscr U_{\mathfrak a,0}=\mathscr I+\mathscr F_H,
\end{equation}
where $\mathscr F_H$ is a Hilbert--Schmidt operator with kernel
$F_H$, i.e.
$$
(\mathscr F_Hf)(x)=\int\limits_0^1F_H(x,s)f(s)ds,\quad f\in
L_2((0,1),\mathbb C^{2r}).
$$
\end{proposition}

\begin{proof}
The proof can be obtained by direct verification.
\end{proof}

\subsubsection{The condition $(A_3)$}
Now let us prove that for all $q\in\mathfrak Q$ the spectral data
$\mathfrak a_q$ satisfy the condition $(A_3)$. In view of
\eqref{UaqI}, this fact directly follows from the following lemma.

\begin{lemma}\label{A3Lemma}
Let $q\in\mathfrak Q$, and let $\mathfrak
a=((\lambda_j,\alpha_j))_{j\in\mathbb Z}$ be any collection
satisfying the asymptotics $(A_1)$. Then
\begin{equation}\label{A3Equiv}
(A_3)\Longleftrightarrow \mathscr U_{\mathfrak a,q}>0.
\end{equation}
\end{lemma}

\begin{proof}
Taking into consideration the relation (\ref{PhiKrel}), we obtain
that
\begin{equation}\label{UaqUa0}
\mathscr U_{\mathfrak a,q}=(\mathscr I+\mathscr K_q)\mathscr
U_{\mathfrak a,0}(\mathscr I+\mathscr K_q^*).
\end{equation}
Since the operator $\mathscr I+\mathscr K_q$ is a homeomorphism of
the space $\mathbb H$, it is enough to prove the equivalence
(\ref{A3Equiv}) only for the case $q=0$.

Since $\mathscr U_{\mathfrak a,0}\ge0$ and the operator $\mathscr
F_H$ in (\ref{UF}) is compact, we obtain that $\mathscr
U_{\mathfrak a,0}>0$ if and only if $\ker\mathscr U_{\mathfrak
a,0}=\{0\}$. Thus it is enough to prove the equivalence
\begin{equation}\label{A3aux}
(A_3)\Longleftrightarrow \ker\mathscr U_{\mathfrak a,0}=\{0\}.
\end{equation}

Set $\mathcal X:=\{ e^{i\lambda_jt}d\ | \ j\in\mathbb{Z},\
d\in\mathrm{Ran}\ \alpha_j \}\subset L_2((-1,1),\mathbb C^r)$ and
note that the condition $(A_3)$ is equivalent to the equality
$\mathcal X^\perp=\{0\}$. Consider the unitary transformation $ U:
L_2((-1,1),\mathbb C^r) \to \mathbb H$ given by
$$
(Uf)(x):= (f(-x),f(x))\in\mathbb C^{2r}, \qquad x\in(0,1).
$$

It follows from the definitions of $\mathscr U_{\mathfrak a,0} $
and $\Phi_0(\lambda)$ that
$$
\ker\mathscr U_{\mathfrak a,0}=\bigcap\limits_{j\in\mathbb
Z}\ker\alpha_j\Phi_0^*(\lambda_j) =(U\mathcal X)^\perp=U\mathcal
X^\perp,
$$
and therefore (\ref{A3aux}) is proved.
\end{proof}

\section{Inverse spectral problem}
In this section we solve the inverse spectral problem for the
operator $T_q$. We show that if a collection $\mathfrak a$
satisfies the conditions $(A_1)-(A_3)$, then $\mathfrak
a=\mathfrak a_q$ for some $q\in\mathfrak Q$ and suggest a method
of constructing such $q$.

\subsection{The Krein accelerant: proof of Theorem \ref{Th3}}
Here we prove Theorem \ref{Th3}, i.e. we show that any collection
$\mathfrak a$ satisfying the conditions $(A_1)$ and $(A_3)$
generates the Krein accelerant belonging to $\mathfrak H_2^s$.

Since the convergence of (\ref{HMudef}) was already proved, it is
left to prove only the following lemma.

\begin{lemma}
Let $\mathfrak a$ satisfy the conditions $(A_1)$ and $(A_3)$,
$\mu:=\mu^{\mathfrak a}$, $H:=H_\mu$. Then the function $H$ is an
accelerant and belongs to $\mathfrak H_2^s$.
\end{lemma}

\begin{proof}
Let us prove that $H$ belongs to $\mathfrak H_2$. It is enough to
prove that the operator $\mathscr I+\mathscr H$ is positive in
$L_2((0,1),\mathbb C^r)$, where $\mathscr H$ is given by
\begin{equation}\label{WienerHopfDef}
(\mathscr Hf)(x)=\int\limits_0^1H(x-t)f(t)dt.
\end{equation}

By virtue of (\ref{UF}) and Lemma \ref{A3Lemma}, the condition
$(A_3)$ implies the \linebreak positivity of the operator
$\mathscr I+\mathscr F_H$ in the space $L_2((0,1),\mathbb
C^{2r})$. Consider the unitary transformation $V:L_2((0,1),\mathbb
C^{2r})\to L_2((0,1),\mathbb C^r)$,
$$
(Vf)(t)=\begin{cases}\sqrt{2}f_2(1-2t),&t\in(0,1/2],\\\sqrt{2}f_1(2t-1),&t\in(1/2,1).\end{cases}
$$
A direct verification shows that $\mathscr I+\mathscr H=V(\mathscr
I+\mathscr F_H)V^{-1}$, and thus the operators $\mathscr
I+\mathscr H$ and $\mathscr I+\mathscr F_H$ are unitary
equivalent. Therefore $I+\mathscr H>0$ in $L_2((0,1),\mathbb
C^r)$. It is left to notice that by virtue of the relation
(\ref{HadjReal}) the function $H$ belongs to $\mathfrak H_2^s$.
\end{proof}

\subsection{Factorization of $\mathscr U_{\mathfrak a,0}$}
Given a collection $\mathfrak a$ satisfying the asymptotics
$(A_1)$, put $\mu:=\mu^{\mathfrak a}$, $H:=H_\mu$, and construct
the operator $\mathscr U_{\mathfrak a,0}$ via (\ref{UaqDef}).  In
this subsection we show that $\mathscr U_{\mathfrak a,0}$ admits a
factorization in $\mathscr G_2(M_r)$. Some statements concerning
the theory of factorization can be found in \ref{add.2}.

\subsubsection{Basic properties of $R_H$}
Recall that for $H\in\mathfrak H_2$ we denote by $R_H$ the
solution of the Krein equation (\ref{KreinEq}). Here we prove some
basic properties of $R_H$.

\begin{lemma}\label{Lemma2}
\begin{itemize}
\item[$(i)$]If $H\in\mathfrak H_2$, then $R_H\in G_2^+(M_r)$
and the mapping
$$
\mathfrak H_2\owns H\mapsto R_H\in G_2^+(M_r)
$$
is continuous;
\item[$(ii)$]if $H\in\mathfrak H_2^s$, then $H^\sharp\in\mathfrak H_2^s$ and
\begin{equation}\label{LemmaRel}
R_{H^\sharp}(\cdot,0)=[R_H(\cdot,0)]^*;
\end{equation}
\item[$(iii)$]the mapping $\Theta:\mathfrak H_2^s\to
\mathfrak Q$ given by $\Theta(H):=iR_H(\cdot,0)$ is continuous;
\item[$(iv)$]if $H\in\mathfrak H_2\cap C^1([-1,1],M_r)$, then $R_H\in
C^1(\Omega^+,M_r)$.
\end{itemize}
\end{lemma}

\begin{proof}
We start from proving $(i)$. Suppose that $H\in\mathfrak H_2$.
Denote by $\mathscr H$ the operator given by
(\ref{WienerHopfDef}), and set $\mathscr H^a:=\chi_a\mathscr
H\chi_a$ (see \ref{add.2}). Since $H\in\mathfrak H_2$,
$\ker(\mathscr I+\mathscr H^a)=\{0\}$ for all $a\in[0,1]$, and the
operator $\mathscr I+\mathscr H^a$ is invertible in the algebra
$\mathscr B(L_2)$ of bounded linear operators acting in
$L_2((0,1),\mathbb C^r)$. Since $\mathscr H^a$ depends continuously
on $a\in[0,1]$, the mapping $[0,1]\owns a\mapsto(\mathscr
I+\mathscr H^a)^{-1}\in\mathscr B(L_2)$ is continuous. Denote by
$\Gamma_{a,H}$ the kernel of the integral operator $-\mathscr
H^a(\mathscr I+\mathscr H^a)^{-1}$. Since $\mathscr H$ is a
Hilbert--Schmidt operator, the mapping
$$
[0,1]\times\mathfrak H_2\owns(a,H)\mapsto\Gamma_{a,H}\in
L_2((0,1)^2,M_r)
$$
is also continuous.

For $(x,t)\in\Omega^+$ put
\begin{equation}\label{bRHrepr}
\hat R_H(x,t):=\int\limits_0^xH(x-y)H(y-t)dy+
\int\limits_0^x\int\limits_0^xH(x-u)\Gamma_{x,H}(u,v)H(v-t)dvdu.
\end{equation}
It is easily seen that the mapping $\mathfrak H_2\owns H\mapsto
\hat R_H\in C(\Omega^+,M_r)$ is continuous. A direct verification
shows that the function
\begin{equation}\label{RHrepr}
R_H(x,t):=\begin{cases}\hat
R_H(x,t)-H(x-t),&(x,t)\in\Omega^+,\\0,&(x,t)\in\Omega^-\end{cases}
\end{equation}
solves the Krein equation (\ref{KreinEq}). Therefore $R_H$ belongs
to $G_2^+(M_r)$, and the mapping $\mathfrak H_2\owns H\mapsto
R_H\in G_2^+(M_r)$ is continuous.

Let us prove $(ii)$. Assume that $H\in\mathfrak H_2^s$. First let
us show that $H^\sharp\in\mathfrak H_2^s$. Construct the integral
operator $\mathscr H^\sharp$ via the formula (\ref{WienerHopfDef})
with $H^\sharp$ instead of $H$. The operators $\mathscr I+\mathscr
H^\sharp$ and $\mathscr I+\mathscr H$ are unitary equivalent under
the unitary transformation $f(t)\mapsto f(1-t)$. Therefore
$I+\mathscr H>0$ if and only if $I+\mathscr H^\sharp>0$, and thus
$H^\sharp\in\mathfrak H_2^s$. The equality (\ref{LemmaRel}) can be
easily verified having noted that
$\Gamma_{a,H}(a-x,a-t)=\Gamma_{a,H^\sharp}(x,t)$ and
$\Gamma_{a,H}(x,t)=[\Gamma_{a,H}(t,x)]^*$ for all $x,t\in[0,a]$
and for all $H\in\mathfrak H_2^s$.

The continuity of $\Theta$ easily follows from its definition and
continuity of the mapping $H\mapsto R_H$, and thus the statement
$(iii)$ is proved.

It is left to prove $(iv)$. It follows from \cite[Chapter
IV]{kreinvolterra} that if $H\in\mathfrak H_2\cap C^1([-1,1],M_r)$,
then the function $a\mapsto\Gamma_{a,H}(u,v)$ is continuously
differentiable for $a\ge\max\{u,v\}$. Therefore taking into
consideration (\ref{RHrepr}) and (\ref{bRHrepr}) we conclude that
$R_H\in C^1(\Omega^+,M_r)$.
\end{proof}

\subsubsection{The GLM equation}
Here we establish structure of the solution of
Gelfand--Levitan--Marchenko (GLM) equation.

\begin{lemma}\label{Th7}
Let $H\in L_2((-1,1),M_r)$. If $H\in\mathfrak H_2^s$, then the GLM
equation
\begin{equation}\label{GLM}
L(x,t)+F_H(x,t)+\int\limits_0^xL(x,s)F_H(s,t)ds=0,\quad
(x,t)\in\Omega^+
\end{equation}
has a unique solution in the class $L_2(\Omega^+,M_{2r})$;
moreover, this solution \linebreak belongs to $G_2^+(M_{2r})$ and
takes the form
\begin{equation}\label{KHform}
L_H(x,t)=\frac{1}{2}\begin{pmatrix}R_H\left(x,\frac{x+t}{2}\right)&R_H\left(x,\frac{x-t}{2}\right)\\
R_{H^\sharp}\left(x,\frac{x-t}{2}\right)&R_{H^\sharp}\left(x,\frac{x+t}{2}\right)\end{pmatrix}.
\end{equation}
\end{lemma}

\begin{proof}
A direct verification shows that the function $L_H$ given by
(\ref{KHform}) solves the GLM equation (\ref{GLM}). Since $F_H\in
G_2(M_{2r})$ and $L_H\in L_2(\Omega^+,M_{2r})$, the results of
\ref{add.2} yield that $L_H\in G_2^+(M_{2r})$.
\end{proof}

\begin{remark}\label{Lrem}
Since the mapping $H\mapsto R_H$ is continuous, it is easily seen that
the mapping $H\mapsto L_H$ given by (\ref{KHform}) is also continuous.
\end{remark}

\subsubsection{Theorem on factorization of $\mathscr U_{\mathfrak a,0}$}
Main result of the present subsection is the following theorem.

\begin{theorem}\label{Th8}
Let $\mathfrak a=((\lambda_j,\alpha_j))_{j\in\mathbb Z}$ be a
collection satisfying the conditions $(A_1)$ and $(A_3)$,
$\mu:=\mu^{\mathfrak a}$, $H:=H_\mu$. Set $q:=\Theta(H)$. Then
\begin{equation}\label{Ua0Fact}
\mathscr U_{\mathfrak a,0}=(\mathscr I+\mathscr K_q)^{-1}(\mathscr
I+\mathscr K_q^*)^{-1},
\end{equation}
where $\mathscr K_q$ is an integral operator with kernel $K_q$
(see Lemma \ref{Th5}).
\end{theorem}

\begin{proof}
By virtue of Lemma \ref{Th7}, the function $L_H$ given by
(\ref{KHform}) solves the GLM equation (\ref{GLM}). Thus, as
follows from \ref{add.2}, the equality
$$
\mathscr U_{\mathfrak a,0}=(\mathscr I+\mathscr L_H)^{-1}(\mathscr
I+\mathscr L_H^*)^{-1}
$$
takes place, where $\mathscr L_H$ is an integral operator with
kernel $L_H$. Therefore it is left to show that $\mathscr
L_H=\mathscr K_q$, i.e. it suffices to show that
\begin{equation}\label{LK}
L_H=K_q,\quad q=\Theta(H).
\end{equation}

Notice that it is enough to prove (\ref{LK}) only for the case
$H\in\mathfrak H_2^s\cap C^1([-1,1],M_r)$. Indeed, the set
$\mathfrak H_2^s\cap C^1([-1,1],M_r)$ is dense everywhere in
$\mathfrak H_2^s$, and the mappings $q\mapsto K_q$, $\Theta$ and
$H\mapsto L_H$ are continuous (see Lemma \ref{Th5}, Lemma
\ref{Lemma2} and Remark \ref{Lrem} respectively).

Let $H\in\mathfrak H_2^s\cap C^1([-1,1],M_r)$. Taking into
consideration Lemma \ref{Th5}, it is easily seen that the equality
(\ref{LK}) is equivalent to the fact that the function
\begin{equation}\label{PhiDef}
\varphi(x,\lambda):=\varphi_0(x,\lambda)+\int\limits_0^x
L_H(x,s)\varphi_0(s,\lambda)ds
\end{equation}
solves the Cauchy problem
\begin{equation}\label{PhiEq}
\vartheta\tfrac{d}{dx}\varphi+\bq\varphi=\lambda\varphi,\quad
\varphi(0,\lambda)=\vartheta a^*.
\end{equation}

Thus it is left to prove (\ref{PhiEq}). Let us introduce the
auxiliary functions
$$
\tH:=\begin{pmatrix}H&0\\0&H^\sharp\end{pmatrix},\quad
\tR_H:=\begin{pmatrix}R_H&0\\0&R_{H^\sharp}\end{pmatrix}.
$$
The definitions of $R_H$ and $\tR_H$ yield that the following
relation holds:
\begin{equation}\label{KreinEq1}
\tR_H(x,t)+\tH(x-t)+\int\limits_0^x\tR_H(x,s)\tH(s-t)ds=0,\quad
(x,t)\in\Omega^+.
\end{equation}
Moreover, by virtue of Lemma \ref{Lemma2} we obtain that $\tR_H\in
C^1(\Omega^+,M_{2r})$.

Noting that
$$
L_H(x,t)=\tfrac{1}{2}\left\{\tR_H\left(x,\tfrac{x+t}{2}\right)+\tR_H\left(x,\tfrac{x-t}{2}\right)J\right\}
$$
and
$$
J\varphi_0(x,\lambda)=\varphi_0(-x,\lambda),\quad
J=\begin{pmatrix}0&I\\I&0\end{pmatrix},
$$
we can rewrite (\ref{PhiDef}) as
$$
\varphi(x,\lambda)=\varphi_0(x,\lambda)+\int\limits_0^x
\tR_H(x,x-s)\varphi_0(x-2s,\lambda)ds.
$$
From this equality, taking into consideration that
$\vartheta\tfrac{d}{dx}\varphi_0(x,\lambda)-\lambda\varphi_0(x,\lambda)=0$,
we easily obtain that
$$
\vartheta\tfrac{d}{dx}\varphi(x,\lambda)+\bq(x)\varphi(x,\lambda)-\lambda\varphi(x,\lambda)
=\{\vartheta
\tR_H(x,0)J\varphi_0(x,\lambda)+\bq(x)\varphi_0(x,\lambda)\}
$$
\begin{equation}\label{SecondInt}
+\int\limits_0^x\left\{\vartheta \tfrac{\partial}{\partial x}[
\tR_H(x,x-s)]+\bq(x)\tR_H(x,x-s)\right\}\varphi_0(x-2s,\lambda)ds.
\end{equation}
Taking into consideration (\ref{LemmaRel}) we conclude that
$\bq(x)=-\vartheta\tR_H(x,0)J$ and thus the relation (\ref{SecondInt}) can
be rewritten as
\begin{eqnarray*}
&\vartheta\frac{d}{dx}\varphi(x,\lambda)+\bq(x)\varphi(x,\lambda)-\lambda\varphi(x,\lambda)
\\
&=\vartheta\int\limits_0^x\left\{\frac{\partial}{\partial x}[
\tR_H(x,x-s)]-\tR_H(x,0)J
\tR_H(x,s)J\right\}\varphi_0(x-2s,\lambda)ds.
\end{eqnarray*}
If we show that
\begin{equation}\label{ZeroRel}
\tfrac{\partial}{\partial x}[ \tR_H(x,x-s)]-\tR_H(x,0)J
\tR_H(x,s)J=0
\end{equation}
for $(x,t)\in\Omega^+$, then (\ref{PhiEq}) will be proved.

Let us show (\ref{ZeroRel}). From (\ref{KreinEq1}) we obtain that
$$
\tR_H(x,x-t)+\tH(t)+\int\limits_0^x\tR_H(x,x-s)\tH(t-s)ds=0,\quad
(x,t)\in\Omega^+,
$$
and differentiating this expression with respect to $x$ we can
write
\begin{equation}\label{DifEq}
\tfrac{\partial}{\partial x}\tR_H(x,x-t)+
\tR_H(x,0)\tH(t-x)+\int\limits_0^x\tfrac{\partial}{\partial x}[
\tR_H(x,x-s)]\tH(t-s)ds=0.
\end{equation}
Now we multiply the relation (\ref{KreinEq1}) by $\tR_H(x,0)J$
from the left and by $J$ from the right, and write
\begin{eqnarray}\nonumber
\tR_H(x,0)J\tR_H(x,t)J+\tR_H(x,0)J\tH(x-t)J
\\\label{MultEq}
+\int\limits_0^x\tR_H(x,0)J\tR_H(x,s)\tH(s-t)Jds=0
\end{eqnarray}
for $(x,t)\in\Omega^+$. Subtracting now (\ref{DifEq}) from
(\ref{MultEq}) and taking into consideration that
$\tH(x)J=J\tH(-x)$, we obtain that the function
$$
X(x,t)=\tfrac{\partial}{\partial x}[ \tR_H(x,x-s)]-\tR_H(x,0)J
\tR_H(x,s)J
$$
solves the equation
$$
X(x,t)+\int\limits_0^xX(x,s)\tH(s-t)ds=0,\quad (x,t)\in\Omega^+.
$$
Since $\tR_H\in C^1(\Omega^+,M_{2r})$, $X\in C(\Omega^+,M_{2r})$
and thus by virtue of Lemma \ref{LemmaIntEq} and the relation
(\ref{KreinEq1}) we conclude that $X(x,t)\equiv0$. Therefore the
relation (\ref{ZeroRel}) follows, and the proof is complete.
\end{proof}

\begin{remark}
Let $\mathfrak a=((\lambda_j,\alpha_j))_{j\in\mathbb Z}$ be a
collection satisfying the conditions $(A_1)$ and $(A_3)$,
$\mu:=\mu^{\mathfrak a}$, $H:=H_\mu$, $q:=\Theta(H)$. Then from
the equalities (\ref{Ua0Fact}) and (\ref{UaqUa0}) we obtain that
\begin{equation}\label{eq1}
\mathscr U_{\mathfrak a,q}=\mathscr I.
\end{equation}
\end{remark}

\subsection{Description of the spectral data: the sufficiency part}
In this subsection we show that if a collection $\mathfrak a$
satisfies the conditions $(A_1)-(A_3)$, then it belongs to
$\mathfrak A$, i.e. that $\mathfrak a=\mathfrak a_q$ for some
$q\in\mathfrak Q$. This is the sufficiency part of Theorem
\ref{Th1}.

We start from proving the following lemma.

\begin{lemma}\label{Lemma5}
Let $\mathfrak a=((\lambda_j,\alpha_j))_{j\in\mathbb Z}$ be a
collection satisfying the conditions $(A_1)-(A_3)$,
$\mu:=\mu^{\mathfrak a}$, $H:=H_\mu$, $q:=\Theta(H)$. Set
$$
\hat P_j:=\Phi_q(\lambda_j)\alpha_j\Phi_q^*(\lambda_j).
$$
Then the collection $\{\hat P_j\}_{j\in\mathbb Z}$ forms a
complete system of pairwise orthogonal projectors.
\end{lemma}

\noindent In order to prove this lemma we need the following
additional statements, that are proved in \cite{sturm} (Lemmas B.2
and B.3 respectively).

\begin{lemma}\label{LemmaPr1}
Let $H$ be a Hilbert space. Assume that $(A_j)_{j=1}^\infty$ is a
sequence in $\mathscr B(H)$ and that $(G_j)_{j=1}^\infty$ is a
sequence of pairwise orthogonal projectors such that the following
statements hold:
\begin{itemize}
\item[(i)] the series $\sum_{j=1}^\infty A_j$ converges in the
strong operator topology to an operator $A$;
\item[(ii)] the orthogonal projector $G:=\mathscr I_H-\sum_{j=1}^\infty G_j$ is of finite rank;
\item[(iii)] $\sum_{j=1}^\infty\|A_j-G_j\|^2<1$ and $\rank A_j\le\rank G_j<\infty$
for every $j\in\mathbb N$.
\end{itemize}
Then $\mathrm{codim}\ \mathrm{Ran}\ A\ge\rank G$.
\end{lemma}

\begin{lemma}\label{LemmaPr2}
Let $H$ be a Hilbert space, and let $\{A_j\}_{j=0}^n$ be a set of
self-adjoint operators from the algebra $\mathscr B(H)$ that are
of finite rank for $j\neq0$. If
$$
\sum\limits_{j=0}^nA_j=\mathscr I_H,\quad \sum\limits_{j=1}^n\rank
A_j\le\mathrm{codim}\ \mathrm{Ran}\ A_0,
$$
then $\{A_j\}_{j=0}^n$ is the set of pairwise orthogonal
projectors.
\end{lemma}

\noindent We use these statements to prove Lemma \ref{Lemma5}.

\begin{prooflemma34}
It follows from Lemma \ref{Lemma3} that the series
$\sum_{j\in\mathbb Z}\hat P_j$ converges in the strong operator
topology, and in view of (\ref{UaqDef}) and (\ref{eq1}) we obtain
that
$$
\sum\limits_{j=-\infty}^\infty\hat P_j=\mathscr I.
$$
Thus, it is enough to show that the operators $\hat P_j$,
$j\in\mathbb Z$ are pairwise orthogonal projectors.

Denote
$$
A_n:=\sum\limits_{\lambda_j\in\Delta_n}\hat P_j.
$$
By virtue of Lemma \ref{Lemma3} we obtain that
$\sum_{n=-\infty}^\infty\|P_{n,0}-A_n\|^2<\infty$, and therefore
there exists an $N_0\in\mathbb N$ such that
$\sum_{|n|>N_0}\|P_{n,0}-A_n\|^2<1$. Moreover, due to the
conditions $\mathrm{(A_1)}$ and $\mathrm{(A_2)}$ we conclude that
$N_0$ can be taken so large that
\begin{equation}\label{eq3}
\sum\limits_{n=-N}^N\sum\limits_{\lambda_j\in\Delta_n}\rank
\alpha_j=(2N+1)r,\qquad N\ge N_0,
\end{equation}
\begin{equation}\label{eq2}
\|\sum\limits_{\lambda_j\in\Delta_n}\alpha_j-I\|<1,\qquad |n|\ge
N_0.
\end{equation}

First let us show that
\begin{equation}\label{eq4}
\sum\limits_{\lambda_j\in\Delta_n}\rank \alpha_j=r,\qquad |n|\ge
N_0.
\end{equation}
Indeed, it follows from (\ref{eq2}) that
$\sum_{\lambda_j\in\Delta_n}\rank \alpha_j\ge r$ and from
(\ref{eq3}) that $ \sum_{\lambda_j\in\Delta_n}\rank
\alpha_j+\sum_{\lambda_j\in\Delta_{-n}}\rank \alpha_j=2r $ for
$|n|\ge N_0$, and thus we obtain (\ref{eq4}).

Fix $N>N_0$ and set
$$
P:=\mathscr I-\sum\limits_{|n|>N}P_{n,0},\quad
A:=\sum\limits_{|n|>N}A_n.
$$
Since $\rank \hat P_j=\rank \alpha_j$ for all $j\in\mathbb Z$
(which follows directly from the definition of $\hat P_j$ and
(\ref{KerRanPhi})), taking into consideration (\ref{eq4}) we
conclude that
$$
\rank A_n\le\sum\limits_{\lambda_j\in\Delta_n}\rank \hat
P_j=\sum\limits_{\lambda_j\in\Delta_n}\rank \alpha_j=r=\rank
P_{n,0},\qquad |n|>N_0.
$$
Recalling also that $\{P_{n,0}\}_{j\in\mathbb Z}$ forms a complete
system of pairwise orthogonal projectors, by virtue of Lemma
\ref{LemmaPr1} we obtain that
$$
\mathrm{codim}\ \Ran A\ge\rank P=(2N+1)r.
$$
Moreover, $ A+\sum_{n=-N}^NA_n=\sum_{j\in\mathbb Z}\hat
P_j=\mathscr I$ and
\begin{eqnarray*}
&\sum\limits_{n=-N}^N\rank
A_n\le\sum\limits_{n=-N}^N\sum\limits_{\lambda_j\in\Delta_n}\rank
\hat P_j
\\
&=\sum\limits_{n=-N}^N\sum\limits_{\lambda_j\in\Delta_n}\rank
\alpha_j=(2N+1)r\le\mathrm{codim}\ \mathrm{Ran}\ A.
\end{eqnarray*}
Therefore, since the operators $A$ and $A_n$, $|n|\le N$, are
self-adjoint, by virtue of Lemma \ref{LemmaPr2} we obtain that the
set
$$
\{\hat P_j\ | \ \lambda_j\in \bigcup\limits_{n=-N}^N \Delta_n\}
$$
is a set of pairwise orthogonal projectors. Since $N$ is
arbitrary, we conclude that projectors $\{\hat P_j\}_{j\in\mathbb
Z}$ are orthogonal ones.
\end{prooflemma34}

In order to prove the sufficiency part of Theorem \ref{Th1} it
obviously suffices to find $q\in\mathfrak Q$ such that $\mathfrak
a=\mathfrak a_q$.

\begin{theorem}\label{Th9}
Let $\mathfrak a=((\lambda_j,\alpha_j))_{j\in\mathbb Z}$ be a
collection satisfying the conditions $(A_1)-(A_3)$,
$\mu:=\mu^{\mathfrak a}$, $H:=H_\mu$, $q:=\Theta(H)$. Then
$\mathfrak a=\mathfrak a_q$.
\end{theorem}

\begin{proof}
It is enough to prove the relation
\begin{equation}\label{PrRel}
\mathrm{Ran}\ \hat P_j\subset\ker(T_q-\lambda_j\mathscr I),\quad
j\in\mathbb Z.
\end{equation}
Indeed, taking into account the completeness of system $\{\hat
P_j\}_{j\in\mathbb Z}$, from (\ref{PrRel}) we immediately conclude
that $\lambda_j(q)=\lambda_j$ for all $j\in\mathbb Z$, where
$\lambda_j(q)$ are the eigenvalues of $T_q$. From this equality
and (\ref{PrRel}) we obtain the relation $P_{j,q}-\hat P_j\ge0$,
$j\in\mathbb Z$, where $P_{j,q}$ are corresponding orthogonal
projectors of $T_q$. However, by virtue of completeness of the
systems $\{\hat P_j\}_{j\in\mathbb Z}$ and
$\{P_{j,q}\}_{j\in\mathbb Z}$ we conclude that $ \sum_{j\in\mathbb
Z}(P_{j,q}-\hat P_j)=0$, and therefore $P_{j,q}-\hat P_j=0$ for
all $j\in\mathbb Z$. Therefore, taking into account Lemma
\ref{Lemma5} and the definition of $\hat P_j$, we conclude that
$$
\Phi_q(\lambda_j)\{\alpha_j(q)-\alpha_j\}\Phi_q^*(\lambda_j)=0,
\quad j\in\mathbb Z,
$$
and by virtue of (\ref{KerRanPhi}) we justify that
$\alpha_j(q)=\alpha_j$, which, together with
$\lambda_j(q)=\lambda_j$, means that $\mathfrak a=\mathfrak
a_{q}$.

Thus it only remains to prove (\ref{PrRel}). Due to the definition
of $\Phi_q(\lambda)$ and (\ref{KerRanPhi}) we obtain that $
\mathrm{Ran}\ \hat P_j=\{\varphi_{q}(\cdot,\lambda_j)\alpha_jc\ |
\ c\in\mathbb C^r\}$. From the other side, by virtue of Lemma
\ref{Lemma1} we obtain that
$$
\ker (T_q-\lambda_j \mathscr I)=\{ \varphi_q(\cdot,\lambda_j)c \ |
\ a\varphi_q(1,\lambda_j)c=0 \}.
$$
Therefore we conclude that it is enough to show that
\begin{equation}\label{FinalRel}
a\varphi_q(1,\lambda_j)\alpha_j=0.
\end{equation}
Let $j,k\in\mathbb Z$ and $c,d\in\mathbb C^r$. Then, taking into
account that $\bq=\bq^*$, $\vartheta^*=-\vartheta$ and integrating
by parts, we obtain that
$$
\lambda_j(\Phi_q(\lambda_j)c\ | \
\Phi_q(\lambda_k)d)=(\vartheta\varphi_q(1,\lambda_j)c\ | \
\varphi_q(1,\lambda_k)d)+\lambda_k(\Phi_q(\lambda_j)c\ | \
\Phi_q(\lambda_k)d),
$$
\begin{equation}\label{Rel1}
\lambda_j\Phi_q(\lambda_k)^*\Phi_q(\lambda_j)-
\lambda_k\Phi_q(\lambda_k)^*\Phi_q(\lambda_j)=\varphi_q(1,\lambda_k)^*\vartheta\varphi_q(1,\lambda_j).
\end{equation}
Since $\hat P_k\hat P_j=0$ if $k\neq j$, we obtain that $
\Phi_q(\lambda_k)\alpha_k\Phi_q^*(\lambda_k)\Phi_q(\lambda_j)\alpha_j\Phi_q^*(\lambda_j)=0$,
and by virtue of (\ref{KerRanPhi}) we conclude that $
\alpha_k\Phi_q^*(\lambda_k)\Phi_q(\lambda_j)\alpha_j=0$.
Multiplying now (\ref{Rel1}) by $\alpha_k$ from the left and by
$\alpha_j$ from the right we obtain that
$$
\alpha_k\varphi_q(1,\lambda_k)^*\vartheta\varphi_q(1,\lambda_j)\alpha_j=0,
$$
and therefore we can write
\begin{equation}\label{Rel2}
\left\{\sum\limits_{\lambda_k\in\Delta_n}(-1)^n\alpha_k\varphi_q(1,\lambda_k)^*\right\}\vartheta
\varphi_q(1,\lambda_j)\alpha_j=0, \qquad \lambda_j\notin\Delta_n.
\end{equation}
Taking into account (\ref{varphiRepr}), it follows from the
Riemann--Lebesgue lemma and the asymptotic behavior of the
sequences $(\lambda_k)$ and $(\alpha_k)$ that
$$
\lim\limits_{n\to\infty}\left\{\sum\limits_{\lambda_k\in\Delta_n}(-1)^n\varphi_q(1,\lambda_k)\alpha_k\right\}=\vartheta
a^*,
$$
and passing to the limit in (\ref{Rel2}) we obtain the relation
(\ref{FinalRel}).
\end{proof}

\subsection{Potential reconstruction: proof of Theorems \ref{Th2} and \ref{Th4}}
Finally, we prove Theorems \ref{Th2} and \ref{Th4}.

\begin{proofth12}
Suppose that $q_1,q_2\in\mathfrak Q$, and let $\mathfrak
a_{q_1}=\mathfrak a_{q_2}$. Let us show that $q_1=q_2$. Write
$\mathfrak a_{q_1}=\mathfrak a_{q_2}=:\mathfrak a$ for short, and
set $\mu:=\mu^{\mathfrak a}$, $H:=H_\mu$. Then by virtue of
Theorem \ref{Th8} the operator $\mathscr U_{\mathfrak
a,0}=\mathscr I+\mathscr F_H$ admits a factorization, and we can
write
$$
\mathscr U_{\mathfrak a,0}=(\mathscr I+\mathscr
K_{q_1})^{-1}(\mathscr I+\mathscr K_{q_1}^*)^{-1}=(\mathscr
I+\mathscr K_{q_2})^{-1}(\mathscr I+\mathscr K_{q_2}^*)^{-1}.
$$
Since any operator may admit at most one factorization of the
above form (see \ref{add.2}), we conclude that
$$
\mathscr K_{q_1}=\mathscr K_{q_2}.
$$

It is left to notice that $\mathscr K_{q_1}=\mathscr K_{q_2}\
\Rightarrow\ q_1=q_2$. Taking into account (\ref{varphiRepr}) we
conclude that $\mathscr K_{q_1}=\mathscr K_{q_2}\ \Rightarrow\
\varphi_{q_1}(\cdot,0)=\varphi_{q_2}(\cdot,0)=:\varphi$, and
therefore we obtain
$$
\vartheta\varphi'+\bq_1\varphi=\vartheta\varphi'+\bq_2\varphi=0,
$$
and thus $\{\bq_1-\bq_2\}\varphi=0$.

Thus it is left to show that for all $x$ the matrix $\varphi(x)$
is invertible. Assume the contrary. Then there exist $x_0\in(0,1]$
and $c\in\mathbb C^r\setminus\{0\}$ such that $\varphi(x_0)c=0$,
and therefore the function $f=\varphi_{q_1}(\cdot,0)c$ is a
non-zero solution of the Cauchy problem $\vartheta f'+\bq_1f=0$,
$f(x_0)=0$. But this is in contradiction with the uniqueness
theorem; thus $\varphi(x)$ is non-singular, and $\bq_1=\bq_2$.

Besides this, by definition of the spectral data we obviously have
$q_1=q_2\ \Rightarrow\ \mathfrak a_{q_1}=\mathfrak a_{q_2}$, and
therefore we conclude that the mapping $\mathfrak Q\owns
q\mapsto\mathfrak a_q\in\mathfrak A$ is bijective.
\end{proofth12}

\begin{proofth14} Theorem \ref{Th4} now directly follows from Theorems
\ref{Th2} and \ref{Th9}.
\end{proofth14}

\section*{Acknowledgements}
The authors are grateful to Rostyslav Hryniv for
helpful discussions and valuable suggestions in preparing this
manuscript.

\appendix

\section{Spaces}\label{add.1}

By $G_2(M_r)$ we denote the set of all measurable functions
$K:[0,1]^2\to M_r$, such that for all $x$ and $t$ in $[0,1]$ the
functions $K(x,\cdot)$ and $K(\cdot,t)$ belong to $L_2((0,1),M_r)$
and, moreover, the mappings
$$
[0,1]\ni x\mapsto K(x,\cdot)\in L_2((0,1),M_r), \quad [0,1]\ni
t\mapsto K(\cdot,t)\in L_2((0,1),M_r)
$$
are continuous on the interval $[0,1]$. It can easily be seen that
$G_2(M_r)\subset L_2([0,1]^2,M_r)$. The set $G_2(M_r)$ becomes a
Banach space upon introducing the norm
$$
\|K\|_{G_2(M_r)}=\max\left\{
\max\limits_{x\in[0,1]}\|K(x,\cdot)\|_{L_2((0,1),M_r)},
\max\limits_{t\in[0,1]}\|K(\cdot,t)\|_{L_2((0,1),M_r)}\right\}.
$$
By $\mathscr G_2(M_r)$ we denote the space of all integral
operators with kernels $K\in G_2(M_r)$. It forms a subalgebra in the
algebra $\mathscr B_\infty$ of compact operators in
$L_2((0,1),\mathbb C^r)$.

We denote
$$
\Omega^+:=\{(x,t) \ | \ 0\le t\le x\le 1\}, \quad \Omega^-:=\{(x,t) \ |
\ 0\le x<t\le 1\}.
$$
We write $G_2^+(M_r)$ for the set of all functions $K\in G_2(M_r)$
such that $K(x,t)=0$ a.e. in $\Omega^-$, and $G_2^-(M_r)$ for set
of all $K\in G_2(M_r)$ such that $K(x,t)=0$ a.e. in $\Omega^+$. By
$\mathscr G_2^\pm(M_r)$ we denote the subalgebra of $\mathscr
G_2(M_r)$ consisting of all operators with kernels $K\in
G_2^\pm(M_r)$.

\section{Factorization of operators}\label{add.2}
\def\thesection{\Alph{section}}

Here we state some well-known facts from the theory of
factorization. In particular, these facts are mentioned in
\cite{sturm}, \cite{MykDirac}. See also \cite{kreinvolterra} for
details.

We say that an operator $\mathscr I+\mathscr F$, $\mathscr F\in
\mathscr G_2(M_r)$ admits a factorization (in $\mathscr G_2(M_r)$)
if there exist $\mathscr L^+\in \mathscr G_2^+(M_r)$ and $\mathscr
L^-\in \mathscr G_2^-(M_r)$ such that
$$
\mathscr I+\mathscr F=(\mathscr I+\mathscr L^+)^{-1}(\mathscr
I+\mathscr L^-)^{-1}.
$$
It is known that if $\mathscr I+\mathscr F$ admits a
factorization, then the corresponding operators $\mathscr L^+$ and
$\mathscr L^-$ are unique. Moreover, the set of operators
$\mathscr F\in\mathscr G_2(M_r)$, such that $\mathscr I+\mathscr
F$ admits a factorization, is open, and the mappings $\mathscr
F\mapsto\mathscr L^\pm\in\mathscr G_2(M_r)$ are continuous.

An operator $\mathscr I+\mathscr F$, $\mathscr F\in \mathscr G_2(M_r)$
admits a factorization if and only if the operators $I+\chi_a\mathscr
F\chi_a$ have trivial kernels for all $a\in[0,1]$. Here $\chi_a$ is an
operator of multiplication by the indicator of the interval $(0,a]$,
i.e.
$$
(\chi_af)(x)=\begin{cases} f(x), & x\in(0,a], \\ 0, & x\in(a,1),
\end{cases}.
$$
If $\mathscr F$ is self-adjoint, then this condition is equivalent to
the positivity of $\mathscr I+\mathscr F$.

From the other side, it is known that $\mathscr I+\mathscr F$ admits a
factorization in $\mathscr G_2(M_r)$ if and only if the equation
\begin{equation}\label{GLMFact}
X(x,t)+F(x,t)+\int\limits_0^xX(x,s)F(s,t)ds=0, \quad
(x,t)\in\Omega^+,
\end{equation}
where $F$ is the kernel of $\mathscr F$, is solvable in
$L_2(\Omega^+,M_r)$. In this case its solution is unique and
belongs to $G_2^+(M_r)$. Eq. (\ref{GLMFact}) is usually
called \emph{the Gelfand--Levitan--Marchenko (GLM) equation}.

Also we formulate the following lemma (see Lemma A.3 in
\cite{sturm}).

\begin{lemma}\label{LemmaIntEq}
Let $F\in L_2((0,1)^2,M_r)$. Then the GLM equation (\ref{GLMFact})
has at most one solution; if (\ref{GLMFact}) is solvable, then the
equation
$$
X(x,t)+\int\limits_0^xX(x,s)F(s,t)ds=0,\quad (x,t)\in\Omega^+
$$
has only trivial solution in $L_2(\Omega^+,M_r)$.
\end{lemma}

\end{document}